\documentclass[11pt]{article}
\usepackage[margin=1.2in]{geometry}
\usepackage{graphicx}
\usepackage{amssymb,amsmath,amsthm}
\usepackage{tikz}
\usetikzlibrary{arrows.meta}
\usepackage{authblk}

\newtheorem{theorem}{Theorem}
\newtheorem{lemma}{Lemma}
\newtheorem{definition}{Definition}

\newcommand{\ceil}[1]{\left\lceil#1\right\rceil}
\newcommand{\set}[1]{\left\{#1\right\}}
\newcommand{\power}[2]{\left(#1\right)^{#2}}

\title{The Speed and Threshold of the Biased Perfect Matching Game}

\date{\today}

\author[1]{Noah Brustle}
\author[1]{Sarah Clusiau}
\author[1]{Vishnu V. Narayan}
\author[2]{Ndiam\'{e} Ndiaye}
\author[1]{Bruce Reed}
\author[3,4]{Ben Seamone\thanks{Emails:\\ \texttt{noah.brustle@mail.mcgill.ca, sarah.clusiau@mail.mcgill.ca,}\\ \texttt{vishnu.narayan@mail.mcgill.ca, ndiame.ndiaye@mail.mcgill.ca,}\\ \texttt{breed@cs.mcgill.ca, seamone@iro.umontreal.ca}}}
\affil[1]{School of Computer Science, McGill University, Montreal, Canada}
\affil[2]{Department of Mathematics and Statistics, McGill University, Montreal, Canada}
\affil[3]{Mathematics Department, Dawson College, Montreal, Canada}
\affil[4]{D\'{e}partement d'informatique et de recherche op\'{e}rationnelle, Universit\'{e} de Montr\'{e}al, Montreal, Canada}

\begin{document}
  \maketitle
\begin{abstract} 
  We show that Maker wins the Maker-Breaker perfect matching game in $\frac{n}{2}+o(n)$ turns when the bias is at least $\frac{n}{\log{n}}-\frac{f(n)n}{(\log{n})^{5/4}}$, for any $f$ going to infinity with $n$ and $n$ sufficiently large (in terms of $f$).
\end{abstract}

\section{Introduction}\label{s:intro}
The (Maker-Breaker) perfect matching game on a graph $G$ with an even number of vertices is played by two players who alternately select edges from $G$, with Breaker choosing first.
Maker wins when she has chosen the edges of a perfect matching. If she never does so, Breaker wins. The (Maker-Breaker) Hamilton cycle game is defined 
analogously. We restrict our attention, and the definition of these games, to graphs $G$ that are cliques. 

Chv\'{a}tal and Erd\H{o}s~\cite{CE78} proved that for sufficiently large $n$ and $G$ a clique on $n$ vertices, 
if both players play optimally then Maker can ensure she wins the Hamilton cycle game in her first $2n$ moves. Since the edges of every cycle with an even number 
of vertices can be partitioned into two matchings this means that Maker will also win the perfect matching game for even $n$.

Chv\'{a}tal and Erd\H{o}s~\cite{CE78} also introduced the biased version of such games, where for some integer $b$, Breaker selects $b$ edges in each turn. They showed that for any positive $\epsilon$, for $n$ sufficiently large, if $b>\frac{(1+\epsilon)n}{\log{n}}$ then, 
Breaker can ensure that he selects all edges incident to one of the vertices. For such values of $b$, Breaker wins both the  biased Hamilton cycle  game and the  biased perfect matching game. Krivelevich~\cite{Kri11} obtained an essentially matching lower bound, showing that for $b<\frac{(1-\epsilon)n}{\log{n}}$  
Maker  wins the biased Hamilton cycle game, and hence also the biased perfect matching game\footnote{A slightly stronger result is stated in \cite{Kri11}, but a careful reading of the proof 
shows that the result stated here is actually what is proven. In particular the use of Lemma 3 stated therein means the approach given there can do no better.}. 

For Maker to win the (unbiased) perfect matching game she must make at least $\frac{n}{2}$ moves as to obtain the edges of the matching. Indeed, Breaker can prevent Maker from just choosing 
the edges of a matching by stealing the last edge, so Maker cannot ensure she wins in fewer than  $\frac{n}{2}+1$ moves. Hefetz et al.~\cite{HKS09} showed that this lower bound is tight; i.e., they proved that when $n$ is sufficiently large Maker can win in $\frac{n}{2}+1$ moves. 

In the same vein, Maker will need at least n+1 moves to win the unbiased Hamilton cycle game, and Hefetz and Stich~\cite{HS09} proved that she can always do so. 

In this paper we focus on the number of moves Maker needs to win the biased perfect matching game. This question has been previously studied by 
Mikala\v{c}ki and Stojakovi\'{c}~\cite{MS17}. They showed that there is some $\delta>0$ such that if $b<\frac{\delta n}{\log{n}}$, then Maker can win in $(1+o(1))\frac{n}{2}$ moves. 
They asked whether this remains true for larger values of $b$. We answer this question in the affirmative by showing:

\begin{theorem}
\label{themaintheorem}
For every $f(n)$ which is $\omega(1)$ and every sufficiently large $n$, if  $b<\frac{n}{\log{n}}-\frac{f(n) n}{(\log{n})^{5/4}}$ then Maker can win the biased perfect matching game in $\frac{(1+o(1))n}{2}$ steps. 
\end{theorem}

We note that this improves the upper bound from \cite{Kri11} on the bias which ensures that Maker can win this game. 

Maker's strategy is a 2-stage approach. In the next section we discuss the goal of each stage and sketch the proof that she can successfully complete them.
In Section \ref{s:details} we fill in the details. 

We close this section with four definitions. At any point in the game, we use $d_B(v)$ to denote the degree of the vertex $v$ in the graph chosen by Breaker and $d_M(v)$ to denote its degree in the graph chosen by Maker. We let $E_B$ and $E_M$ be the edges picked by Breaker and Maker.

\section{A Proof Sketch}\label{s:sketch}

Since Maker can always pretend Breaker has chosen more edges than he actually has, we can and do assume that $f(n) =o(\log{\log{n}})$. As noted above, Maker's approach has two stages. 

At the end of the first stage, Maker has chosen a subgraph which contains a  matching  $M$ such that $|V(M)|=(1-o(1))\frac{n}{2}$. 
At this point  every vertex $v$ satisfies $d_M(v) \le  3$ and if $v$ is in $M$, $d_M(v)=1$.  Thus 
Maker chooses at most $\frac{n}{2} +o(n)$ edges in the first stage. In the second stage Maker chooses only edges within $V(G)-V(M)$. 
The edges chosen in the first stage ensure that  Maker can make these choices so that after $14|V(G)-V(M)|$  turns, her  graph will contain a matching with vertex set $V(G)-V(M)$ and hence she has won the game in $\frac{n}{2} +o(n)$ turns. 

Now, if Maker chose no edges within $V(G)-V(M)$ during the first stage, then the second stage simply consists of playing the perfect matching game  on a clique with $n'=o(n)$ vertices. But $b=\omega(\frac{n'}{\log{n'}})$, hence Breaker wins this game and Maker will not be able to successfully complete the second stage. Thus,
during the first stage, as well as choosing the edges of $M$, Maker must choose edges between the elements of $V(G)-V(M)$ so that the graph $F$ formed by these non-matching edges has a structure that allows her to carry out the second stage. 
To specify this structure precisely we need some definitions.

We say that a tree is {\it matchable} if it has a perfect matching. 

We say a rooted  tree $T$ with root $r$ is {\it augmenting} if every node at an odd depth has exactly one child. Thus every leaf is at even depth and there is a unique matching $N_T$
in T, with vertex set T-r containing, for each vertex at odd depth, the edge from this vertex to its unique child.  We note that for every leaf $x$ of $T$, letting $P_x$ be the path of $T$ from $x$ to $r$,  we have that $N_x=(N_T-E(P_x)) \cup (E(P_x)-N_T)$ is a matching of $T$  with vertex set $V(T)-x$.

For some even $p$, $F$ will have $p$ components that are not matchable trees and each of these will be augmenting.  This means that in the second stage Maker need only construct a matching $N$,  with $\frac{p}{2}$ edges 
containing exactly one leaf in each component that is an augmenting tree. Since $T-x$ has a matching for every component  $T$ of $F$ and leaf $x$ of $T$, we can extend $N \cup M$ to the desired perfect matching. 

We
need to impose some additional structure on the augmenting trees in order to ensure we can construct $N$. Specifically, 
that each of these trees has a large number of leaves, and all of these leaves have low degree in Breaker's graph. We say a node v is {\it troublesome} if $d_B(v)>\frac{n}{\sqrt{\log{n}}}$. We let $\ell=\lceil  \sqrt{f(n)}(\log{n})^{1/4} \rceil $. 
We say that a tree  $T$ is {\it nice} if it is augmenting, has $\ell$ leaves; none of which is troublesome,  has maximum degree 3, and  all of its nontroublesome nonleaf nodes at even depth have two children.

Modifying  the approach of Krivelevich, we shall show:  

 \begin{lemma}
\label{stage2lemma}
Given that Maker has  chosen a graph each of whose components is a nice tree, matchable, or an edge and such that the number of components that are nice trees  is an even $p$ lying between $\frac{n}{\sqrt{\log{n}}}$ and $\frac{4n}{\sqrt{\log{n}}}+1$, she has a strategy that allows her to construct a matching with $p/2$ edges  whose vertex set contains exactly one leaf from each nice tree in $14p$  moves.

\end{lemma}

So, to complete the proof we need only show that Maker has a strategy that ensures she can construct a graph  as in the hypothesis of Lemma \ref{stage2lemma} in $\frac{n}{2}+o(n)$ turns.  

 If no vertex ever became troublesome, then as long as there are  at least $\frac{n}{\sqrt{\log{n}}}+2$ unmatched vertices we could simply pair two unmatched vertices to increase the size of 
 the matching. Then the remaining unmatched vertices could be  the roots of the nice trees and we could build these trees, each with $\ell$ leaves, by repeatedly choosing two edges (in two consecutive turns) 
  from a leaf of such a tree to two vertices in different edges  of the matching,  and also add the two matching edges these vertices are in to the tree.  This increases the number of leaves of the tree by one (see Figure \ref{fig:strategy}). Since there are no troublesome vertices, there will always be plenty of choices 
 for the two matching edges. We note that each tree would have exactly $4\ell-3$ vertices so the total size of our trees would be $O(p\ell)=o(n)$. 
 \begin{figure}[ht]
    \centering
    \begin{tikzpicture}[scale=0.5]
        \begin{scope}[every node/.style={circle, fill=black, draw, inner sep=0pt,
        minimum size = 0.15cm
        }]
            
            \node[] (t01) at (2,10) {};
            \node[] (t02) at (1,9) {};
            \node[] (t03) at (3,9) {};
            \node[] (t04) at (1,8) {};
            \node[] (t05) at (3,8) {};
            \node[] (t06) at (0,7) {};
            \node[] (t07) at (2,7) {};
            \node[] (t08) at (0,6) {};
            \node[] (t09) at (2,6) {};
            
            \node[] (t11) at (3.5,5) {};
            \node[] (t12) at (2.5,4) {};
            \node[] (t13) at (4.5,4) {};
            \node[] (t14) at (2.5,3) {};
            \node[] (t15) at (4.5,3) {};
            
            \node[] (t21) at (7,9) {};
            \node[] (t22) at (6,8) {};
            \node[] (t23) at (8,8) {};
            \node[] (t24) at (6,7) {};
            \node[label={[label distance=4]245:$v$}] (t25) at (8,7) {};
            
            \node[] (m01) at (12,10) {};
            \node[] (m02) at (14,10) {};
            \node[] (m11) at (12,9) {};
            \node[] (m12) at (14,9) {};
            \node[label={[label distance=2]165:$u$}] (m21) at (12,8) {};
            \node[label={[label distance=2]15:$w$}] (m22) at (14,8) {};
            \node[] (m31) at (12,7) {};
            \node[] (m32) at (14,7) {};
            \node[] (m41) at (12,6) {};
            \node[] (m42) at (14,6) {};
            \node[] (m51) at (12,5) {};
            \node[] (m52) at (14,5) {};
            \node[] (m61) at (12,4) {};
            \node[] (m62) at (14,4) {};
            \node[label={[label distance=2]195:$y$}] (m71) at (12,3) {};
            \node[label={[label distance=2]345:$z$}] (m72) at (14,3) {};
            \node[] (m81) at (12,2) {};
            \node[] (m82) at (14,2) {};
        
        \end{scope}

        \begin{scope}[every edge/.style={draw=black}]
            
            \path[thick] (t01) edge node {} (t02);
            \path[thick] (t01) edge node {} (t03);
            \path[ultra thick, dotted] (t02) edge[draw=blue] node {} (t04);
            \path[ultra thick, dotted] (t03) edge[draw=blue] node {} (t05);
            \path[thick] (t04) edge node {} (t06);
            \path[thick] (t04) edge node {} (t07);
            \path[ultra thick, dotted] (t06) edge[draw=blue] node {} (t08);
            \path[ultra thick, dotted] (t07) edge[draw=blue] node {} (t09);
            
            \path[thick] (t11) edge node {} (t12);
            \path[thick] (t11) edge node {} (t13);
            \path[ultra thick, dotted] (t12) edge[draw=blue] node {} (t14);
            \path[ultra thick, dotted] (t13) edge[draw=blue] node {} (t15);
            
            \path[thick] (t21) edge node {} (t22);
            \path[thick] (t21) edge node {} (t23);
            \path[ultra thick, dotted] (t22) edge[draw=blue] node {} (t24);
            \path[ultra thick, dotted] (t23) edge[draw=blue] node {} (t25);
            
            \path[ultra thick, dotted] (m01) edge[draw=blue] node {} (m02);
            \path[ultra thick, dotted] (m11) edge[draw=blue] node {} (m12);
            \path[ultra thick, dotted] (m21) edge[draw=blue] node {} (m22);
            \path[ultra thick, dotted] (m31) edge[draw=blue] node {} (m32);
            \path[ultra thick, dotted] (m41) edge[draw=blue] node {} (m42);
            \path[ultra thick, dotted] (m51) edge[draw=blue] node {} (m52);
            \path[ultra thick, dotted] (m61) edge[draw=blue] node {} (m62);
            \path[ultra thick, dotted] (m71) edge[draw=blue] node {} (m72);
            \path[ultra thick, dotted] (m81) edge[draw=blue] node {} (m82);
            
            \path[very thick, dashed] (t25) edge node {} (m21);
            \path[very thick, dashed] (t25) edge node {} (m71);
            
        \end{scope}

        \begin{scope}[every node/.style={draw=none,rectangle}]
            
            \node (Tlabel) at (7,7.5) {$T$};
            \node (Mlabel) at (15.5,6) {$M$};
            
        \end{scope}
    \end{tikzpicture}
    \caption{If no vertex becomes troublesome, Maker's strategy first selects a matching $M$ and then builds a collection of trees by choosing edges from a leaf to two matching edges. In the figure, matching edges are represented by dotted lines. Maker chooses edges from the leaf $v$ of tree $T$ to the edges $uw$ and $yz$. Thus leaf $v$ of $T$ is replaced by leaves $w$ and $z$.}
    \label{fig:strategy}
\end{figure}
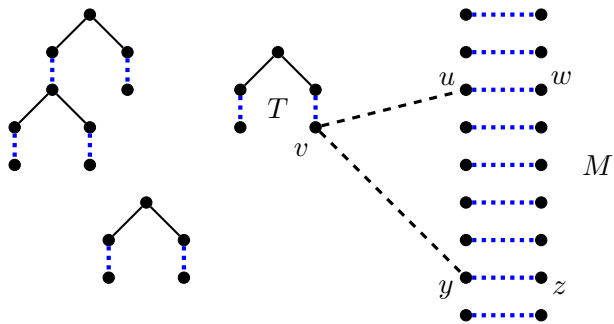
 
 Troublesome vertices complicate our choices in two ways. There are fewer choices for the edges out of them, and we want to deal with the troublesome vertices before Breaker picks all the 
 edges incident to them. This may mean we cannot pick an edge we otherwise would have picked because we have to choose an edge with a troublesome endpoint instead.
 Thus, we may have to root a tree of size exceeding two at a troublesome vertex even though there are many more than  $\frac{n}{\sqrt{\log{n}}}$ unmatched vertices because Breaker has picked 
 all the edges between the troublesome vertex and the unmatched vertices. 
 This will increase the number of trees of size exceeding two  we create.
 Furthermore, Breaker may have chosen all the edges 
 from a troublesome leaf of one of the trees we want to make nice to vertices in our matching. This will force us to pick an edge 
 from this leaf to a singleton tree to construct a matchable tree with four or more vertices. 
 This explains our need for both matchable and nice trees of size exceeding two.  
 
 The reason we can handle the trouble caused by troublesome vertices is that there are not very many of them. We carry out fewer than $n$ turns and Breaker chooses fewer than $\frac{n}{\log{n}}$ edges 
 in each turn. Hence, the total degree in Breaker's graph is at most $\frac{2n^2}{\log{n}}$ and there are at most $\frac{2n}{\sqrt{\log{n}}}$ troublesome vertices. We exploit this fact repeatedly. 

For any set S we let $\tau(S)$ be the number of troublesome
vertices 
in $S$.  We say a tree $T$  is {\it small} if $|V(T)|<4(\tau(V(T))+\ell)$. 

We shall show: 
\begin{lemma}
\label{stage1lemma}
Maker has a strategy to choose, in an initial set of turns, a graph F  which,  for some even $p$ with  $\frac{n}{\sqrt{\log{n}}} \le p \le \frac{4n}{\sqrt{\log{n}}}+1$, has $p$ components that are  small nice trees and whose remaining components are small matchable trees each of which either contains a troublesome vertex or is an edge.
\end{lemma}

We note that $F$ has at most $p$ components that are not an edge and contain no troublesome vertex.
So, $F$ has at most $\frac{2n}{\sqrt{\log{n}}}+p \le  \frac{6n}{\sqrt{\log{n}}}+1$ components  that are not edges.  Each such component $T$ is small, i.e. $|V(T)| \le 4(\tau(V(T)+\ell)$. Since there are at most $\frac{2n}{\sqrt{\log{n}}}$ troublesome vertices, the total number of vertices in components of $F$ that are not edges is at most 
$\frac{24n\ell}{\sqrt{\log{n}}}+4\ell+\frac{8n}{\sqrt{\log{n}}}$.

Since this is $o(n)$, Maker constructs $F$ in $\frac{n}{2}+o(n)$ turns and so combining
Lemmas \ref{stage2lemma} and \ref{stage1lemma} proves Theorem \ref{themaintheorem}. 

This completes our intuitive sketch of the approach. It  remains to  set out the precise strategy used in each stage 
 and prove that both  can be successfully applied.  We do so in the next section.
 
\section{The Details} \label{s:details}

\subsection{Stage 1 algorithm formalization}

  To begin we set out the invariants that Maker will ensure hold throughout Stage 1 and her strategy for doing so.
 We remark that Maker may decide during 
 one turn what to do in both that turn and the next. Because of this, each step of the algorithm will take either 1 turn or 2 turns. 

After $i$ turns in Stage 1,  Maker will ensure that the graph she  has chosen consists of a matching $M_i$, a set $q_i$ 
of small matchable trees, each with a troublesome root, 
and for some $p_i > \frac{n}{\sqrt{\log{n}}}$,  a family ${\cal T}_i=\{T_1,...,T_{p_i}\}$ of augmenting trees, each of maximum degree three. Furthermore, 
She will attempt to also ensure that every non-leaf non-troublesome node at even depth of one of the augmenting trees has degree 2.
However, when adding children underneath a nontroublesome node she will add two, in two consecutive turns,  and between these turns this will not be true.

We have ${\cal T}_0=\set{T_1,...,T_n}$,$q_i=0$, $p_i=n$ where each $T_i$ is a singleton, and  $E_M=E_B=\emptyset$.

\begin{enumerate}
    \item \textbf{If no tree in ${\cal T}_i$ has a leaf that is troublesome:}
    \begin{enumerate}
        \item \textbf{If $p_i>2\ceil{\frac{n}{\sqrt{\log{n}}}}+1$ and there are $u,v$ such that $\deg_M(u)=\deg_M(v)=0$ and $(u,v)\notin E_B$}
        
        Maker chooses $e=(u,v)$.  She has removed 2 trees and added an edge to the matching, so $p_{i+1}=p_i-2$,  ${\cal T}_{i+1}={\cal T}\setminus \set{\set{u},\set{v}}$ and $M_{i+1}=M_i\cup \set{e}$.
        
        \item \textbf{Otherwise if there is $T_j\in {\cal T}_i$ with fewer than  $\ell$ leaves}
        
        Maker picks a leaf $v$ of $T_j$, and finds   an edge $(w,z)\in M_i$ such that neither $w$ nor $z$ is troublesome and $(v,w)\notin E_B$. Maker chooses $(v,w)$ and adds this edge and $e$ to $T_j$. 
        
        Then, after Breaker's turn, Maker finds  another edge $(x,y)\in M_i$ such that neither $x$ nor $y$ is troublesome and $(v,x)\notin E_B$, chooses $(v,x)$ and adds this edge and $(x,y)$  to $T_j$. 
        
        The number of trees in ${\cal T}_i$  hasn't changed in either step so $p_{i+2}=p_{i+1}=p_i$ however,$M_{i+1}=M_i\setminus\set{w,z}$ and $M_{i+2}=M_{i+1}\setminus\set{(x,y)}$.

        \item \textbf{Otherwise}
        
        We end the algorithm.
    \end{enumerate}
    \item \textbf{Otherwise, consider  a tree $T_j \in {\cal T}_i$ and   troublesome leaf $v$  of  $T_j$  where the pair $(j,v)$ is chosen so as to maximize  $d_B(v)$}
    \begin{enumerate}
        \item[(a)] 
            If possible, Maker chooses an edge $vw$ such that $w$ is contained in an edge  $(w,z)$ of $M_i$ such that neither $w$ nor $z$ is troublesome and $(v,w)\notin E_B$. She adds $vw$ and $wz$ to  $T_j$. Set  $M_{i+1}=M_i\setminus{(w,z)}$ and  $p_{i+1}=p_i$.
        \item[(b)] Otherwise, Maker chooses an edge $vw$ such that $d_M(w)=0$, $w$ is not troublesome, and $vw \not\in E_B$ and  adds it to $T_j$. 
           She then adds $T_j$ to the set of matchable trees containing a troublesome vertex  and deletes $T_j$ and $z$  from ${\cal T}_i$. So $p_{i+2}=p_i-2$,  ${\cal T}_i ={\cal T}\backslash \{T_j,w\}$, $M_{i+2}=M_i$, $q_{i+2}=q_i+1$.
    \end{enumerate}

\end{enumerate}
We need to show that  this strategy  shows that  Lemma \ref{stage1lemma} holds. 

\subsection{Proof of Lemma \ref{stage1lemma}}

We will  show that Maker can  always choose an edge with the desired properties; hence, that
the first stage ends. We  then demonstrate  that when  the stage  terminates,  the  properties
required by Lemma \ref{stage1lemma} hold. Before doing either we  make some useful observations 
about properties that hold throughout the process. 

We note first that  as discussed above, we  have fewer than $\frac{2n}{\sqrt{\log{n}}}$ troublesome vertices.

Each non-singleton tree of any ${\cal T}_i$ is composed of some edges that were in the matching 
at some point and others that were not. The first edge chosen from the root of such a tree 
was never in the matching and was chosen in Case 2 a) or Case 1 b). If this edge was chosen in Case 2 a) then the 
root is troublesome so there are never more than $\frac{2n}{\sqrt{\log{n}}}$ such trees.
If this edge was chosen in 1 b), then in the iteration it was chosen Maker did not carry out 1 a). 
This means that there were at most  $  \lceil \frac{2n}{\sqrt{\log{n}}} \rceil$ singleton trees at this point
as otherwise for every singleton tree  consisting of a (nontroublesome) vertex $u$ there would be at least  $\lceil \frac{n}{\sqrt{\log{n}}} \rceil$ 
choices  for a $v$ such that we could apply 1 a) for $u$ and $v$.

In every iteration, the set of  non-troublesome roots of the nonsingleton trees that are not in the matching, are contained in the set of the last 
$\lceil \frac{2n}{\sqrt{\log{n}}} \rceil$ vertices to be in singleton trees. 

Hence, we always have that the total number of  trees of size at least three in Maker's graph 
is at most $\frac{4n}{\sqrt{\log{n}}}+1$. 

The size of a tree of ${\cal T}_i$  can be increase in Case 1 b) or Case 2. Each iteration of 1 b) increases the number of leaves 
in the tree by 1, thus, we may only have $\ell-1$ such turns per tree. Now, 
in Case 1 b)  4 vertices  are added to a tree; while in Case 2, at most two vertices are added.
Furthermore, every time we carry out Case 2 a), we add a troublesome internal vertex to the tree, and we carry out Case 2 b) for each tree at most once. We have that at every point in the algorithm each of the nonsingleton trees not in the matching is small. 

 Combining this with the last paragraph, and the fact that there are at most $\frac{2n}{\sqrt{\log{n}}}$ troublesome 
 vertices,  we have that  the union of these trees has size at most
$\frac{16\ell n}{\sqrt{\log{n}}}+4l+\frac{8n}{\sqrt{\log{n}}}$.

\begin{lemma}
Maker is always able to choose an edge as the strategy 1 algorithm specifies she should. 
\end{lemma}

\begin{proof}
$~$

   Case 1:  
   
    At most $2b$ vertices can become troublesome in any turn. We need only show if we cannot carry out 1a) but there is some tree $T_j$ of ${\cal T}_i$ with fewer than $2\ell$ leaves, then for any leaf $v$ of $T_j$ there are at least $3b+2$ edges of $M_i$ that contain no troublesome vertices and Breaker has not isolated from $v$.  This implies the two turn process is possible, as Breaker  can only choose $b$ edges from $v$ and make $2b$ vertices troublesome in her intervening turn. 
    
    By our bound on the size of the non-singleton trees and number of troublesome vertices, there are at least $n 
    -\frac{16\ell n} {\sqrt{\log{n}}}-4\ell-\frac{12n} {\sqrt{\log{n}}}$ vertices that are either singletons or in  matching edges not containing a troublesome vertex. As shown above, since we did not carry out 1a), there are at most $\lceil \frac{2n}{\sqrt{\log{n}}}\rceil $ singletons. Since $v$ is not troublesome, there are most $\frac{n}{\sqrt{\log{n}}}~ w$ such that  $vw  \in E_B$. So there are  indeed the desired  $3b+2$ matching edges.

   Case 2:
   
       We first show that if for every $v$ from which we choose an edge in the first turn of a Case 2 step, we have $d_B(v)< n-\frac{16\ell n}{\sqrt{\log{n}}}-4\ell-\frac{20n}{\sqrt{\log{n}}}$, then Maker will always be able to find an appropriate edge in each turn of a  Case 2
       step. We then prove that we can ensure this upper bound on $d_B(v)$ holds. 
       
       The first step is similar to  the analysis for Case 1. We know that  there are at least $\frac{16\ell n}{\sqrt{\log{n}}}+4l+\frac{20n}{\sqrt{\log{n}}}$
       vertices of $G$ that are not joined to $v$ by an edge of Breaker's graph. At least $\frac{12n}{\sqrt{\log{n}}}$ of these are singletons or in a matching edge.  Since  there are at most $\frac{2n}{\sqrt{\log{n}}}$ troublesome vertices, it follows that we can carry out one of Step 2 a) or Step 2 b).   
 
       It remains to show that whenever Maker attempts to choose an edge from a vertex $v$ in the first turn of a Caser 2 step, we have $d_B(v)< n -    \frac{16\ell n}{\sqrt{\log{n}}}-4l-\frac{20n}{\sqrt{\log{n}}}$. We assume for a contradiction this is not the case and consider 
       the sequence of her (one or two turn) steps  culminating with the first  turn in which this bound does not hold.

       We look at the suffix of this sequence starting immediately after the last step where  a choice was made via an application of  Case 1. We let $k$ be the number of  steps  in this suffix. Since, we are in Case 2 for each of these $k$ steps, we have $k \le \frac{2n}{\sqrt{\log{n}}}$ steps. Furthermore, each of these steps consists of only one turn.  We let $a_i$ be the vertex that Maker makes into a non-leaf  in the  $i^{th}$ of these $k$ turns/steps and note that the $a_i$ are distinct. For each $i$, we let 
        $d^j_i = d_B(a_i) - |\{a_l| l \ge j,a_ia_l \in E_B\}|$, after $j$ steps of the sequence.
        
        Now, Maker carried out a step before this sequence of $k$ steps, as at the start of the game there are no troublesome vertices. Before this step there were no troublesome vertices and there were at most two turns in the step, each of   which  can increase a Breaker degree by at most 2b. So we end up with $ d^0_i \le  \frac{n}{\sqrt{\log n}}+\frac{2n}{\log n}$. 

        We define the potential function $p(j)$, as follows:

\begin{align*}
p(j) := \frac{1}{k-j+1}\sum_{i=j}^{k}d^{j-1}_i.
\end{align*}

Note that, we are assuming that $p(k) = d^{k-1}_k \ge n -\frac{16\ell n}{\sqrt{\log{n}}}-4l-\frac{20n}{\sqrt{\log{n}}}$. On the other hand
\begin{align*}
p(1)  = \frac{1}{k}\sum_{i=1}^{k}d^0_i \le \frac{1}{k} \times k \times (\frac{n}{\sqrt{\log n}}+\frac{2n}{\log n})=\frac{n}{\sqrt{\log n}}+\frac{2n}{\log n}.
\end{align*}

We derive a contradiction by bounding the increase in the potential function in each step. 
By definition $\sum_{i=j+1}^{k}d^j_i-\sum_{i=j+1}^{k}d^{j-1}_i$ is the sum of the number of 
edges Breaker chose in the $j^{th}$ turn with exactly one end in $\{a_{j+1},...,a_k\}$ 
and the number of edges he has chosen between $a_j$ and   $\{a_{j+1},...,a_k\}$  by the end of the 
$j^{th}$ turn. The first of these is at most $b$ and the second is at most $k-j$. 

We obtain: 

\begin{align*}
(*) \sum_{i=j+1}^{k}d^j_i \le \sum_{i=j+1}^{k}d^{j-1}_i + b+k-j.
\end{align*}

By our definition of $d^j_i$ and choice of $a_i$ to maximize the degree in the Breaker's graph,
at the start of the $j^{th}$ turn,
\begin{align*}
\forall j\le i \le k,  d^{j-1}_i \le d_B(a_i) \le d_B(a_j) \le d^{j-1}_j+(k-j).
\end{align*}
 
Summing up and dividing by $k-j+1$, we obtain
\begin{align*}
d^{j-1}_j \ge \frac{\sum_{i=j}^{k}d^{j-1}_i}{k-j+1} -(k-j).
\end{align*}
 
Thus, noting $\sum_{i=j+1}^{k}d^{j-1}_i = \sum_{i=j}^{k}d^{j-1}_i + d^{j-1}_j$:
\begin{align*}
\sum_{i=j+1}^{k}d^{j-1}_i\le \frac{k-j}{k-j+1}\sum_{i=j}^{k}d^{j-1}_i+(k-j).
\end{align*}
 
Combining this with $(*)$ yields: 
\begin{align*}
\frac{\sum_{i=j+1}^{k}d^{j}_i}{k-j} \le \frac{\sum_{i=j}^{k}d^{j-1}_i}{k-j+1}+\frac{b}{k-j}+2.
\end{align*}
Hence 
\begin{align*}
p(k) \le p(1)+ 2k+b\sum_{r=1}^{k} \frac{1}{k} \le O(\frac{n}{\sqrt{\log{n}}})+b(\log{n}).
\end{align*}

Since $b<\frac{n}{\log{n}}-\frac{f(n)n}{(\log{n})^{5/4}}$ and $\ell=\lceil \sqrt{f(n)} \rceil =o(f(n))$,
the desired result follows. 
\end{proof}

It remains to show that the set of trees created by the end of Stage 1 have the properties required 
by Lemma \ref{stage1lemma}. 

Note that $p$ is necessarily even upon completion of the phase - as $n$ is even, the set of vertices in the matchable trees is even, and by construction all trees in $t$ have an odd number of vertices,
so we must have an even number of trees.

Clearly there are no leaves of any tree in ${\cal T}_i$  that are troublesome when the stage terminates,  
as termination only occurs in Case 1. Furthermore, 
since Case 1 b)  was not carried out in the last iteration every tree in ${\cal T}_i$ has 
$\ell>1$ leaves.

Our construction also ensures:
\begin{enumerate}
    \item all our leaves  are at even depth (except in the middle of a 2-turn step),
\item  whenever we add a node at odd depth to the tree  we always add exactly one child along with it, 
\item  whenever we make a  node at even depth into a nonleaf we add exactly two children underneath it if it is nontroublesome(Case 1b)) and one vertex if it is troublesome (Case 2), and 
\item we never add children under a nonleaf except in the step in which it becomes a nonleaf.  
\end{enumerate}

Thus, at the end of the Stage, ${\cal T}_i$ consists of an even number 
$p$ of nice trees. It remains to bound $p$.

 We have shown above, that there are always at most  $\frac{4n}{\sqrt{\log n}}+1$ non-singleton trees. Thus,  at termination, $|{\cal T}_i| \le \frac{4n}{\sqrt{\log n}}+1$. 
 
 It remains to show that $p_i$ is always at least $\frac{n}{\sqrt{\log{n}}}$.  We note that 
 $p_i$ can only be reduced in Cases 1a) or 2b) and, never by more than 2. Now, Case 1a) only  applies  if $p_i > \frac{2n}{\sqrt{\log n}}$
 so  $p_i$ cannot be reduced below $\frac{n}{\sqrt{\log{n}}}$.
Furthermore, as shown above, when we carry out Case 2b) by choosing an edge from some $v$, there are at least $\frac{16\ell n}{\sqrt{\log{n}}}+4l+\frac{20n}{\sqrt{\log{n}}}$ vertices $w$ such that $vw$ is not an edge of Breaker's graph. 
As noted above, at least  $\frac{8n}{\sqrt{\log{n}}}$ of these are in singleton trees or matching edges. 
Further, by our bound on the number of troublesome vertices,
fewer than $ \frac{4n}{\sqrt{\log{n}}}$ of these are in matching edges that contain a troublesome vertex.
Finally no other matching edge contains one of these vertices as otherwise we would have carried out Case 2a).
It follows that there are at least $ (\frac{4n}{\sqrt{\log n}})$ singleton trees whenever we carry out Case 2b). 
Thus, we obtain $ p_i > \frac{4n}{\sqrt{\log n}}-1$ upon completion of the phase. So Lemma \ref{stage1lemma} has been proved.

\subsection{Proof of Lemma \ref{stage2lemma}}

We  now present and analyze  the strategy that Maker will use in the second phase: Strategy 2. 

  In  doing so we consider an auxiliary multigraph $F'$, whose structure depends on a
  specific graph $F$  satisfying the hypotheses of Lemma \ref{stage2lemma} that Maker has constructed.  For each component $T_i$ of $F$ that is a nice tree, we let $S_i$ be the set  $\ell$ leaves of $T_i$.  $F'$ has vertices $v_1,...,v_p$. The number of edges 
  between $v_i$ and $v_j$ is the number of unchosen edges between $S_i$ and $S_j$. We note that there are at most $\ell^2$ edges between 
  two vertices  of $F'$ and since no vertex in any $S_i$ is troublesome,  the minimum degree of $F'$ is at least:
  
  $\ell^2(p-1)-\frac{\ell n}{\sqrt{\log{n}}} > f(n){\sqrt{\log{n}}}(\frac{n}{\sqrt{\log{n}}}-1) - \frac{\ell n}{\sqrt{\log{n}}} >
  \frac{f(n) n}{4}$ which is  $\omega(n)$.
  Thus, although $b=\omega( \frac{|V(F')|}{\log{|V(F')|}} )$ we have that $b=o(\frac{\delta(F')}{\log{\delta(F')}})$. This allows us to  adapt the strategy of \cite{Kri11} to show that Maker can choose  a Hamilton cycle  in $F'$ in $14p$ moves.
  This Hamilton cycle contains a  perfect matching corresponding to a matching containing exactly one leaf from each component of $F$ which is a nice tree, so we have proven Lemma \ref{stage2lemma}. 
  
  As noted in \cite{Kri11}, it is enough to present a random strategy that creates a Hamilton cycle with positive probability in 
  $14p$ steps.  For if Maker cannot ensure he always creates a Hamilton cycle in $14p$ turns, then Breaker can ensure he never 
  does so. 
  
We consider applying the  2 phase strategy given below to $F'$. Note that at the start neither Maker nor Breaker
has chosen any edges of $F'$. We define $d_B$ and $d_M$ as before. When Maker chooses an edge, she assigns it a direction, we 
let $d^+_M(v)$ be the outdegree of $v$ in the resultant digraph and define the {\it danger} of $v$ to be $d_B(v)-2bd^+_M(v)$. 
\vskip0.2cm

Strategy 2: 
\vskip0.2cm
Phase 1: While there is a vertex whose outdegree in the Maker digraph is less than $10$, Maker  chooses such a vertex $v$ so as to maximize its  danger.
If possible Maker chooses an edge out of $v$ uniformly at random from all the unchosen edges incident to $v$, to add to Maker's digraph. 
Otherwise, Maker concedes defeat. 
\vskip0.2cm
Phase 2: We let $P$ be a longest path in Maker's graph chosen if possible to induce a cycle. If there is a cycle through $V(P)$ then Maker stops, 
winning if this is a Hamilton cycle and losing otherwise. Otherwise,  Maker looks for a path $Q$ with the same vertex set as $P$ such 
that some edge between the endpoints of $Q$ has not been chosen. If she finds such a path she chooses the edge between its endpoints, otherwise she 
concedes defeat. 
\vskip0.4cm

We let $M^*$ be the graph chosen by Maker at the end of Phase 1. We will show: 

\begin{lemma}
\label{anotherlemma1}
With positive probability $M^*$ is connected and for any longest path $P$ of any supergraph $M'$ of $M^*$, 
either $P$ is Hamiltonian and $M'$ has a Hamiltonian cycle or 
the set $\{e|e \in E(F'),~ s.t.~ M'+e ~contains~a~cycle ~on ~V(P)\}$ has size at  least $\frac{11pn}{\log{n}}+1$. 
\end{lemma}

Now, there are at most $10p$ turns in the first phase. Furthermore, 
for every turn in the second phase in which Maker chooses an edge $e$, if $V(P)$ 
is not the vertex set of a component, then the length of the longest path 
in Maker's graph increases as in the new Maker's graph there  is a path using 
the vertices of $P$ and any vertex joined to $P$ by an edge. So, if $M^*$ is 
connected, then the second phase has at most $p$ turns and if it stops because 
$V(P)$ is a cycle, then this is a Hamiltonian cycle. 

So, if $M^*$ is connected, Breaker chooses at most $\frac{11pn}{\log{n}}$ edges in total.
Hence, if in addition for any longest path $P$ of any supergraph $M'$ of $M^*$ 
either $P$ is Hamiltonian and $M'$ has a Hamiltonian cycle or 
the set $\{e|e \in E(F'),~ s.t.~ M'+e ~contains~a~cycle ~on ~V(P)\}$ has size at  least $\frac{11pn}{\log{n}}+1$
then Phase 2 must terminate with the construction of a Hamiltonian cycle. 

Thus Lemma \ref{anotherlemma1} implies Lemma \ref{stage2lemma}, and it remains to prove it.

\subsection{The Proof of Lemma \ref{anotherlemma1}}
\label{thedetails}

The first  key result is similar to one given by Gebauer and Szabo in \cite{GS09}. The only differences are 
that we consider a multigraph, we use $d^+_M$ instead of $d_M$, we have a slightly weaker bound on $b$, and we 
obtain a weaker bound on $d_B$. This latter fact significantly simplifies the proof which uses 
a potential function argument, similar to, but slightly more complicated than  
that used in the proof of Lemma \ref{stage1lemma}.

\begin{lemma}
\label{lowdegreelemma} 
For sufficiently large C, If Maker applies  Strategy 2 to an input satisfying the hypotheses of Lemma \ref{stage2lemma}, then throughout the algorithm  every  vertex of $F'$ with 
$d^+_M(v) <10$ satisfies $d_B(v) < 3n$.
\end{lemma}

\begin{proof} 
 We assume for a contradiction that after some number $k$ of moves in the first phase, there is a vertex $v$ 
 with $d_M^+(v)<10$ and $d_B(v)>3n$ this implies that the danger of $v$ exceeds $\frac{5n}{2}$.
 
 We let $a_i$ be the vertex from which Maker chooses an edge in the $i^{th}$ turn and let 
 $A_i$ be the set consisting of the vertices appearing in $\{a_i,...,a_k\}$.  We note this is a set 
 even though $\{a_i,...,a_k\}$ may be a multiset. 
 
 We let the potential after $i-1$ turns be the average danger of the elements in $A_i$.
 We note that Breaker's $i^{th}$ turn increases this potential by at most $\frac{2b}{|A_{i+1}|} $
 We note that if $A_i$ and $A_{i+1}$ are the same set (i.e. if $a_i=a_j$ for some $j>i$),
 then Maker's $i^{th}$ turn decreases this potential by exactly $\frac{2b}{|A_i|}$.
 Moreover, since $a_i$ has maximum danger in $A_i$, Maker's turn never increases 
 the potential.

 Thus, unless $|A_i|>|A_{i+1}|$, the potential does not increase in the $i^{th}$ turn.
 Furthermore, the total increase in the potential in turns such that  $|A_i|>|A_{i+1}|$ is at most
 \begin{align*}
     \sum_{k=1}^{|A_1|} \frac{2b}{k} \le \sum_{k=1}^n \frac{2b}{k} \le b(\log{n}) \le 2n.
 \end{align*}
 The desired result follows.
 \end{proof}

Now, $F'$ has minimum degree $\omega(n)$, so this lemma implies that every random choice of a directed edge that  Maker makes 
is not that different from simply choosing a random vertex of $F'$ as an out-neighbour. This us allows us to use standard techniques to 
show the following two lemmas.

\begin{definition}
For any $S \subseteq V(F')$, we define $N(S)=N_{M^*}(S)$ to be those vertices outside of S joined to some vertex of S by an edge 
of $M^*$. 
\end{definition}

  \begin{lemma}
  \label{expansion}
  If Maker adopts   Strategy 2, then with probability $1-o(1)$, 
  for every  $S \subseteq V(F')$, if $|S| \le \frac{p}{2}$ then $N(S)$ is nonempty
  while if  $|S| \le \frac{p}{100}$,
  then $|N(S)| >2|S|$. 
  \end{lemma}

\begin{proof}
Again we modify an argument of Krivelevich. We will give an upper bound to the probability of having $S$ and $N(S)=A$ such that $S\cup A \subseteq V(F')$, $A\cap S=\emptyset$, and $A$ is small enough for the conclusion of the lemma to be false. In order to do so, we will use the union bound on that probability for fixed $S$ and $N(S)$. 

Fix $j$ and $k$. Then, by using the Stirling's approximation to bound $\binom{p}{j+k}$, the number of choices for $S$ and $A$ such that $|S|=j$ and $|A|=k$ and $|S\cap A|=0$ is:
\[\binom{p}{j+k}\binom{j+k}{j}
    \le \binom{p}{j+k}2^{j+k}
    \le \power{\frac{p}{j+k}}{j+k} (2e)^{j+k}\]

Noting that $\forall v\in S$ Maker chooses 10 edges in the first phase from $v$.
Since none of the leaves corresponding to $v$ in the first phase were troublesome, by Lemma \ref{lowdegreelemma} the number of unchosen edges out of $v$ when we choose those edges is at least $p\ell^2-\frac{\ell n}{\log{n}}-3n$. Hence, since $\ell=\ceil{\sqrt{f(n)}\power{\log n}{\frac{1}{4}}}=\omega(1)\cdot \power{\log n}{\frac{1}{4}}$ and $p=\Theta\left(\frac{n}{\sqrt{\log{n}}}\right)$ the probability that an edge we pick is from $v$ to a vertex in $S\cup A$ is at most 
\begin{align*}
  \frac{\ell^2(|S|+|A|)}{pl^2-\frac{\ell n}{\log{n}}-3n}= \frac{j+k}{p-\frac{n}{\omega(1)\cdot \power{\log n}{\frac{1}{4}}\log{n}}-3\frac{n}{\omega(1)\cdot \sqrt{\log n}}}=\frac{j+k}{p-o(p)}\le  \frac{j+k}{\frac{9p}{10}}=\frac{j+k}{p}\frac{10}{9}.
\end{align*}
  
  Noting that we are only interested in the cases that violate the conclusion of the lemma, so $\frac{j+k}{p}<\frac{9}{10}$. 
  
  Given that we are choosing $10j$ out edges out of our set $S$ during Strategy 2, the probability that we make choices so that $N(S) \subseteq  A$ is less than
   \begin{align*}
   \power{\frac{j+k}{p}}{10j}
   \power{\frac{10}{9}}{10j}
   \leq
   \power{\frac{j+k}{p}}{8j} \power{\frac{10}{9}}{8j}.
   \end{align*}
   
  Thus, by the union bound, the probability that there exists a set $S$ of $j$ nontroublesome vertices of $V(F')$ such that $|N(S)|=k$ is bounded above by: 
  
  \begin{align*}
  \power{\frac{j+k}{p}}{7j-k} \power{\frac{10}{9}}{8j}(2e)^{j+k}.
  \end{align*}
  
  Recall that the sum of the $t$ first terms of a geometric sequence defined by $a_0$ and $a_{i+1}=\lambda a_i$ with $\lambda\in (0,1)$ is: $a_0\cdot \frac{1-\lambda^{t+1}}{1-\lambda}=a_0\cdot O(1)$. We will now study 2 cases, when $j\leq \frac{p}{100}$ and when $j>\frac{p}{100}$. 
  
  If we have a set $S$ of size $j \le \frac{p}{100}$ that violates the conclusion of the lemma, then the size of $N(S)$ must be $k<2j$, so the probability is bounded above by: 

  \begin{align*}
  \power{\frac{j+k}{p}}{j} \power{\frac{1}{25}}{3j} \power{\frac{200e}{81}}{4j} \le \power{\frac{j+k}{p}}{j} \power{\frac{1}{5}}{j}
  \end{align*}
  Assume that $i=(j+k)$ is fixed, then since $k\leq 2j$, we can conclude that $i\geq j\geq \frac{i}{3}$. Considering just one vertex of $S$ we know that $i=|S \cup N(S)|\geq 11$ so we must also have $j\geq 4$. So by summing over $j$ we get:
 
 \[\sum_{j=\max\left\{4,\ceil{\frac{i}{3}}\right\}}^{i} \power{\frac{i}{5p}}{j}=\power{\frac{i}{5p}}{\max\left\{4,\ceil{\frac{i}{3}}\right\}}\cdot O(1)\]
 
 Noting that $\power{\frac{i}{5p}}{\max\left\{4,\ceil{\frac{i}{3}}\right\}}$
 is less than $\frac{1}{25p^2}$ when $i\leq \sqrt{p}$ and less than $\power{\frac{1}{25}}{\sqrt{p}}$ otherwise since $j+k\leq 4j\leq \frac{p}{25}$ by considering whether $i$ is greater or less than $\sqrt{p}$, by summing over $i$ we get:
 
 \[\sum_{i=11}^{\frac{p}{25}} \power{\frac{i}{5p}}{\ceil{\frac{i}{3}}}\cdot O(1)\leq  \frac{p}{25}\left(\frac{1}{25p^2}  +\power{\frac{1}{25}}{\sqrt{p}}\right)\cdot O(1)=o(1)\]
  
So, the probability of the conclusion failing for some $S$ such that $|S|=j$, $j\leq \frac{p}{100}$ is $o(1)$. 

Otherwise, $\frac{p}{100} \le  j \le \frac{p}{2}$ and the only case for $k$ that violates the conclusion of the lemma is $k =0$. In this case, the probability that there is some set $S$ of $j$ nontroublesome vertices of $V(F')$ such that $N(S)=\emptyset$ can be bounded above by: 
\begin{align*}
    &\power{\frac{j+k}{p}}{7j-k} \power{\frac{10}{9}}{8j} (2e)^{j+k}\\
    = &\power{\power{\frac{j}{p}}{7} \power{\frac{10}{9}}{8} 2e}{j}\\
   \le &\power{\frac{20e}{9}\cdot \power{\frac{5}{9}}{7}}{j}\\
   \le &\power{\frac{1}{100}}{j}
\end{align*}
Given our bounds on $j$, by summing over $j$ and using the sum of the first terms of a geometric sequence, we get at most $\power{\frac{1}{100}}{\frac{n}{100\sqrt{\log n}+100}}\cdot O(1)$ which is $o(1)$.

\end{proof}

  \begin{lemma}
  \label{usingexpansion}
  If 
  every  $S \subseteq V(F')$ with  $|S| \le \frac{p}{100}$ satisfies  $|N(S)| >2|S|$
  then the following holds: 
  
  (*) for any endpoint $u$ of any  longest path $P$ of any supergraph $M'$ of $M^*$ there are at least $\frac{p}{100}$
  vertices $w$ such that $u$ and $w$ are the endpoints of a path on $V(P)$. 
  \end{lemma}
 
 Lemma \ref{usingexpansion} is an immediate corollary of Lemma 6.3.3 of \cite{HKSSbook}, which follows from a well-known lemma of P\'{o}sa \cite{Pos76}.
  
 Now, if for every  $S \subseteq V(F')$, with $|S| \le \frac{p}{2}$, $N(S)$ is nonempty, then every component of $M^*$
  contains more than half its vertices so $M^*$ is connected. Furthermore, if (*) holds, for any longest 
  path $P$ of any supergraph $M'$ of $M^*$, there are at least $\frac{p}{100}$ vertices that are endpoints of paths on $V(P)$.
  Applying (*) to each of these paths we see that furthermore, if (*) holds, there are $\frac{p^2}{20000}$ pairs of vertices that form the 
  endpoints of paths on $V(P)$. By our lower bound on the degrees in $F'$,
  there must be $\frac{\ell^2p^2}{20000} -\frac{p\ell n}{\sqrt{\log{n}}}>\frac{11pn}{\log{n}}+1$ edges of $F'$ that join such a pair of vertices. 
  So, Lemmas \ref{expansion} and \ref{usingexpansion} imply Lemma \ref{anotherlemma1}.

\bibliography{perfect-matching}{}

\begin{thebibliography}{1}

\bibitem{CE78}
V.~Chvátal and P.~Erdös.
\newblock Biased positional games.
\newblock In B.~Alspach, P.~Hell, and D.J. Miller, editors, {\em Algorithmic
  Aspects of Combinatorics}, volume~2 of {\em Annals of Discrete Mathematics},
  pages 221 -- 229. Elsevier, 1978.

\bibitem{GS09}
Heidi Gebauer and Tibor Szabó.
\newblock Asymptotic random graph intuition for the biased connectivity game.
\newblock {\em Random Structures \& Algorithms}, 35(4):431--443, 2009.

\bibitem{HKSSbook}
Dan Hefetz, Michael Krivelevich, Milo{\v{s}} Stojakovi{\'{c}}, and Tibor
  Szab{\'o}.
\newblock {\em The Hamiltonicity Game}, pages 75--83.
\newblock Springer Basel, Basel, 2014.

\bibitem{HKS09}
Dan Hefetz, Michael Krivelevich, Miloš Stojaković, and Tibor Szabó.
\newblock Fast winning strategies in maker–breaker games.
\newblock {\em Journal of Combinatorial Theory, Series B}, 99(1):39 -- 47,
  2009.

\bibitem{HS09}
Dan Hefetz and Sebastian Stich.
\newblock On two problems regarding the hamiltonian cycle game.
\newblock {\em The Electronic Journal of Combinatorics}, 16(R28), 2009.

\bibitem{Kri11}
Michael Krivelevich.
\newblock The critical bias for the hamiltonicity game is (1+o(1))n/ln n.
\newblock {\em Journal of the American Mathematical Society}, 24(1):125--131,
  2011.

\bibitem{MS17}
Mirjana Mikalački and Miloš Stojaković.
\newblock Winning fast in biased maker-breaker games.
\newblock {\em Electronic Notes in Discrete Mathematics}, 61:863 -- 868, 2017.
\newblock The European Conference on Combinatorics, Graph Theory and
  Applications (EUROCOMB'17).

\bibitem{Pos76}
L.~Pósa.
\newblock Hamiltonian circuits in random graphs.
\newblock {\em Discrete Mathematics}, 14(4):359 -- 364, 1976.

\end{thebibliography}
\bibliographystyle{plain}

\end{document}